\documentclass[12pt]{amsart}
\usepackage{amsfonts}
\usepackage{amscd}
\usepackage{amsmath, amsfonts,amssymb,amsthm,mathrsfs,xypic}
\usepackage{graphicx}
\xyoption{curve}
\usepackage{fancybox}
\numberwithin{equation}{section}

\theoremstyle{plain}
\newtheorem{theorem}[equation]{Theorem}
\newtheorem{proposition}[equation]{Proposition}
\newtheorem{lemma}[equation]{Lemma}
\newtheorem{corollary}[equation]{Corollary}
\newcommand{\Hom}{\operatorname{Hom}}
\theoremstyle{definition}
\newtheorem{definition}[equation]{Definition}

\theoremstyle{remark}
\newtheorem{remark}[equation]{Remark}

\newcommand{\ul}{\underline}
\newcommand{\cok}{\mathrm{cok}}

\newcommand{\To}{\longrightarrow}

\def\C{\mathcal{C}}

\def\Ho{\mathsf{Ho}}
\def\D{\mathcal{D}}

\def\A{\mathcal{A}}
\def\T{\mathcal{T}}
\def\X{\mathcal{X}}
\def\Y{\mathcal{Y}}

\begin{document}
\title [\tiny{Priles of one-sided triangulated categories}]{Priles of one-sided triangulated categories and exact model categories }
\author [\tiny{Zhi-Wei Li}] {Zhi-Wei Li}
\thanks{\tiny{Date: July 28, 2014.}}
\thanks{\tiny{The author is supported by JSNU12XLR025.}}


\begin{abstract}
We introduce the notion of a prile of one-sided triangulated categories. Roughly speaking, a prile  consists of two one-sided triangulated categories having a common full subcategory which inherits a pretriangulated structure from these ambient categories. The main example arises from exact model categories. This allows us to recover the pretriangulated structure of the homotoy category of an exact model category.

\end{abstract}

\maketitle

\setcounter{tocdepth}{1}

\section{Introduction}

The main purpose of this paper is to construct a pretriangulated category from a left triangulated category and a right triangulated category. This is motivated by the study of exact model categories.

The concept of an exact model category was introduced in \cite{Hovey02}. In brief, it is a weakly idempotent complete additive category which has both an exact structure and a model structure and the two structures are compatible. Exact model categories play more and more important roles in the study of algebraic triangulated categories \cite{Hovey02, Beligiannis01}, singularities categories \cite{Becker12}, relative homological algebras \cite{Beligiannis/Reiten07}, approximation theory \cite{Saorin/Stovicek} and algebraic geometry \cite{Gillespie12, Gillespie13}. For more background and overview about exact model categories we refer the reader to \cite{Hovey07, Stovicek}.

The foundational result of exact model categories is Hovey's correspondence \cite{Hovey02}: The model structure on a weakly idempotent complete exact category $\A$ that respects the exact structure is equivalent to two compatible complete cotorsion pairs on $\A$. Given an exact model category $\A$, let $\A_c, \A_{triv}, \A_f$ be the classes of cofibrant, trivial and fibrant objects, respectively. Denoted by $\omega=\A_c\cap\A_{triv}\cap \A_f$ and $\A_{cf}=\A_c\cap\A_f$, Hovey's correspondence implies that the homotopy category $\Ho(\A)$ of $\A$ is equivalent to the stable category $\A_{cf}/\omega$; see \cite{Beligiannis/Reiten07, Gillespie11, Becker12}.

The above interpretations motivate us to construct the homotopy theory of an exact model category $\A$ from its exact structure. Recall that D. Quillen has built the homotopy theory of a model category from its model structure;  \cite{Quillen67, Quillen69}. When $\A$ is a hereditary exact model category \cite{Gillespie11}, its homotopy theory can be constructed from its exact structure. In this case, the subcategory $\A_{cf}$ is a Frobenius category with $\omega$ as the injective-projective objects and the stable category $\A_{cf}/\omega$ has a triangulated structure induced by $\omega$ by Happel's Theorem \cite{Happel88} and the equivalence between $\A_{cf}/\omega$ and $\Ho(\A)$ is a triangle equivalence \cite{Becker12}.

For general case, our starting point is to characterise which kind of additive subcategories of one-sided triangulated categories inherit the corresponding one-sided triangulated structures. This is motivated by the results in \cite{ZWLi} where the author has established one-sided triangulated structures on $\A_f/\omega$ and $\A_c/\omega$ from $\omega$. We introduce the notations of a {\it  left coreflective functor} and a {\it  right reflective functor}. A left coreflective functor is a fully faithful functor from an additive category to a left triangulated category which has a right adjoint and satisfies additional condition; see Definition 2.3. A right reflective functor is defined dually. We prove that if $i:\X\to \C$ is a left coreflective functor then $\X$ has an induced left triangulated structure from $\C$ and if $i$ is a right reflective functor then $\X$ has an induced right triangulated structure from $\C$; see Proposition 2.4. In order to construct a pretriangulated structure on $\A_{cf}/\omega$ from the one-sided triangulated structures on $\A_f/\omega$ and $\A_c/\omega$, a {\it prile of one-sided triangulated categories} is defined. It is a triple $(\C_1, \X, \C_2)$ consisting of a left coreflective functor from $\X$ to $\C_1$ and a right reflective functor from $\X$ to $\C_2$ satisfying some compatible conditions; see Definition 2.6. The result about priles is the following, see Theorem 2.7 for details.
\vskip5pt
\noindent{\bf Theorem A: }  {\it \ Let $(\C_1, \X, \C_2)$ be a prile of one-sided triangulated categories. Then $\X$ inherit a pretriangulated structure from $\C_1$ and $\C_2$. }

\vskip5pt

We apply the above results to the one-sided triangulated categories $\A_f/\omega$ and $ \A_c/\omega$ and get the expected results:

\vskip5pt
\noindent{\bf Theorem B: }{\it \ Let $\A$ be an exact model category.
\vskip 5pt
$(i)$ \ \ $(\A_f/\omega, \A_{cf}/\omega,
 \A_c/\omega)$ is a prile of one-sided triangulated categories.
 \vskip5pt
 $(ii)$ \ $\A_{cf}/\omega$ has a pretriangulated structure inherit from $\A_f/\omega$ and $\A_c/\omega$.

 \vskip5pt
 $(iii)$ \ The pretriangulated structure on $\A_{cf}/\omega$ in $(ii)$ is equivalent to the one on $\Ho(\A)$ constructed by Quillen in \cite{Quillen67}.}
\vskip5pt

The statement $(i)$ is  Theorem 4.7, $(ii)$ is Corollary 4.8 and $(iii)$ is Corollary 4.10.

The contents of the paper are as follows. In Section 2, we introduce the notion of a prile of one-sided triangulated categories and prove Theorem A. In Section 3, we recall the standard material of exact model categories and their homotopy categories. Section 4 is aimed to prove Theorem B.

\section{Priles of one-sided triangulated categories}

\subsection{One-sided triangulated categories} \ Let $\C$ be an additive category and $\Omega$ an additive covariant endofunctor on $\C$. Let $\triangle$ be a class of left triangles of the form $\Omega(A)\stackrel{h}\to C\stackrel{g}\to B\stackrel{f}\to A$. Morphisms between left triangles are triples $(\gamma, \beta,\alpha)$ of morphisms which make the following diagram commutative:
\[
\xymatrix{
\Omega(A)\ar[r]^h\ar[d]_{\Omega(\alpha)} & C\ar[r]^{g}\ar[d]_{\gamma} & B\ar[d]_{\beta} \ar[r]^{f} & A\ar[d]_\alpha  \\
\Omega(A') \ar[r]^{h'} & C'\ar[r]^{g'} & B' \ar[r]^{f'} & A'.}
\]

\begin{definition} \cite[Definition 2.2]{Beligiannis/Marmaridis94} \ The pair $(\Omega, \triangle)$ is called a {\it left triangulated structure} on $\C$  if $\triangle$ is closed under isomorphisms and satisfies the following four axioms:
\vskip5pt
${\rm(LT1)}$ \ For any morphism $f: B\to A$ there is a left triangle in $\triangle$ of the form $\Omega(A)\to C\to B\stackrel{f}\to A$. For any object $A\in \C$, the left triangle $0\to A\stackrel{1_A}\to A\to 0$ is in $\triangle$.
\vskip5pt
${\rm (LT2)}$  (Rotation axiom)\ For any left triangle $\Omega(A)\stackrel{h}\to C\stackrel{g}\to B\stackrel{f}\to A$ in $\triangle$, the left triangle $\Omega(B)\stackrel{-\Omega(f)}\to\Omega(A)\stackrel{h}\to C \stackrel{g}\to B$ is also in $\triangle$.
\vskip5pt
${\rm (LT3)}$ \ Every diagram of the form
\[
\xymatrix{
\Omega(A)\ar[r]^h \ar[d]_{\Omega(\alpha)}& C\ar[r]^{g}\ar@{.>}[d] & B\ar[d]_\beta \ar[r]^{f} & A\ar[d]^\alpha  \\
\Omega(A') \ar[r]^{h'} &C'\ar[r]^{g'}& B' \ar[r]^{f'} & A'}
\]
whose rows are in $\triangle$ can be completed to a morphism of triangles.
\vskip5pt
${\rm (LT4)}$ (Octahedral axiom)\ Given two composable morphisms $g:C\to B $ and $f:B\to A$, there is a commutative diagram
\[
\xymatrix{
& \Omega(F)\ar[d]^{m\Omega(l)} &  & \\
\Omega(B)\ar[r]^m\ar[d]_{\Omega(f)} & D\ar[r]^{k}\ar[d]^{\alpha} & C\ar@{=}[d] \ar[r]^{g} & B\ar[d]^f  \\
\Omega(A) \ar[r]^{n}\ar@{=}[d] & E\ar[r]^{h}\ar[d]^\beta & C \ar[r]^{fg} \ar[d]^{g}& A\ar@{=}[d]\\
\Omega(A)\ar[r]^{i} & F\ar[r]^{l}& B\ar[r]^{f}& A}
\]
such that the rows and the second column from the left are triangles in $\triangle$.
\end{definition}

A {\it left triangulated category } is a triple $(\C, \Omega, \triangle)$ consisting of an additive category $\C$ and a left triangulated structure $(\Omega, \triangle)$ on $\C$. Sometimes, we suppress $\Omega$ and $\triangle$ and say that $\C$ is a left triangulated category.

 Dually, we can define the notions of a {\it right triangulated structure} on an additive category and  a {\it right triangulated category }.

\vskip5pt
In topology, the idea of an one-sided triangulated category even goes back to Heller's work in \cite{Heller60}. It is used by D. Quillen to define the homotopy theory of a model category; see \cite{Quillen67, Quillen69}. In algebra, one-sided triangulated categories arising from additive categories have also been studied by many authors; see \cite{Keller/Vossieck87, Beligiannis/Marmaridis94, Beligiannis01, Beligiannis/Reiten07, Liu/Zhu13}.

\begin{remark} \ If the endofunctor $\Omega$ is an autoequivalence, the left triangulated category $(\C, \Omega, \triangle)$ is a triangulated category in the sense of Verdier; see \cite{Verdier}.
\end{remark}

\subsection{Left coreflective and right reflective functors}

\begin{definition} Let $(\C, \Omega, \triangle)$ be a left triangulated category and $\X$ an additive category. We call a fully faithful functor $i: \X\to \C$ {\it left coreflective} if the following two conditions hold:
\vskip5pt
$(i)$ $i$ has a right adjoint $Q$ (called the {\it coreflector}).
\vskip5pt
$(ii)$ Given any morphism $f:B\to A$ in $\C$, let $\Omega(A)\to C\to B\stackrel{f}\to A$ and $\Omega iQ(A)\to D\to iQ(B)\stackrel{iQ(f)}\to iQ(A)$ be the left triangles in $\triangle$ corresponding to $f$ and $Q(f)$, respectively. Then $Q(t)$ is an isomorphism where the morphism $t$ is given by applying (LT3) of Definition 2.1 to the following commutative diagram of triangles:
\[
\xymatrix{
\Omega iQ(A)\ar[r] \ar[d]_{\Omega(\varepsilon_{_A})}& D\ar[r]\ar[d]_t & iQ(B)\ar[d]_{\varepsilon_{_B}} \ar[r]^{iQ(f)} & iQ(A)\ar[d]^{\varepsilon_{_A}}  \\
\Omega(A) \ar[r] &C\ar[r]& B \ar[r]^{f} & A}
\]
(where $\varepsilon:iQ\to {\rm id}_\C$ is the counit of the adjunction $(i,Q)$).
\end{definition}

If in addition, $\X$ is a subcategory of $\C$, we will call it a {\it left coreflective subcatebory} of $\C$.

Dually, one can define the notions of a {\it right reflective functor} and a {\it right reflective subcategory} of a right triangulated category.
\vskip5pt
For example, if $L:\C\to \C$ is an exact localization functor between triangulated categories, the subcategory $\ker L$ is left coreflective and ${\rm Im} L\to \C$ is right reflective; see \cite{BIK1}.

\vskip5pt
Let $i:\X\to \C$ be a left coreflective functor with $\C=(\C, \Omega, \triangle)$. Let $Q$ be the coreflector of $i$. A left triangle $Q\Omega i(G)\to E\to F\to G$ in $\X$ is said to be {\it distinguished}, if it is isomorphic to $Q\Omega i(A)\stackrel{Q(h)}\to Q(C)\stackrel{\eta^{-1}_{_B}Q(g)}\To B\stackrel{f}\to A$ for some $f:B\to A$ in $\X$ with $\Omega i(A)\stackrel{h}\to C\stackrel{g}\to i(B)\stackrel{i(f)}\to i(A)\in \triangle$ (where $\eta: {\rm id}_\X\to Qi $ is the unit of the adjunction $(i,Q)$, it is a natural isomorphism; see \cite[Theorem IV.3.1]{MacLane98}). We use $\triangle_\X$ to denote the class of distinguished left triangles in $\X$.


Dually, let $j:\Y\to \D$ be a right reflective functor with $\D=(\D, \Sigma, \bigtriangledown)$. Let $R$ be the {\it reflector}, i.e. a left adjoint of $j$. Then, we can define a class of distinguished right triangles $\bigtriangledown_\Y$ in $\Y$.

\vskip5pt
\begin{proposition} \ $(i)$ \ Let $i:\X\to \C$ be a left coreflective functor with $\C=(\C, \Omega, \triangle)$. Then $(Q\Omega i, \triangle_\X)$ is a left triangulated structure on $\X$.
\vskip5pt
$(ii)$ Let $j:\Y\to \D$ be a right reflective functor with $\D=(\D, \Sigma, \bigtriangledown)$. Then $(R\Sigma j, \bigtriangledown_\Y)$ is a right triangulated structure on $\Y$.
\end{proposition}
\begin{proof}\  $(i)$ \ We check the axioms of Definition 2.1 one by one.

  (LT1) \ For any morphism $f: B\to A$ in $\X$, by the construction of $\triangle_\X$, there is a left triangle $Q\Omega i(A)\to Q(C)\to B\stackrel{f}\to A\in \triangle_\X$. For any object $A\in \X$, $0\to Qi(A)\stackrel{\eta^{-1}_{_A}}\to A\to 0\in \triangle_{\X}$ and then $0\to A\stackrel{1_A}\to A\to 0\in \triangle_{\X}$ since $\eta_{_A}$ is an isomorphism.

(LT2) \ By the construction of $\triangle_\X$, it is enough to check that $$Q\Omega i(B) \stackrel{-Q\Omega i(f)}\To Q\Omega i(A)\stackrel{Q(h)}\to Q(C)\stackrel{\eta_{_B}^{-1}Q(g)}\to B \in \triangle_\X$$
for $Q \Omega i(A)\stackrel{Q(h)}\to Q(C)\stackrel{\eta_{_B}^{-1}Q(g)}\to B\stackrel{f}\to A\in \triangle_\X $.
Embedding the morphism $iQ(g)$ into a left triangle $\Omega iQi(B)\stackrel{m}\to D\stackrel{l}\to iQ(C)\stackrel{iQ(g)}\to iQi(B)$ in $\triangle$, there is a commutative diagram of left triangles in $\triangle$:
\[
\xymatrix{
\Omega  iQi(B)\ar[r]^{\ \ \ \ m} \ar[d]_{\Omega(\varepsilon_{_{i(B)}})}& D\ar[r]^l\ar[d]_t & iQ(C)\ar[d]_{\varepsilon_{_C}} \ar[r]^{iQ(g)} & iQi(B)\ar[d]^{\varepsilon_{_{i(B)}}}  \\
\Omega i(B) \ar[r]^{-\Omega i(f)} & \Omega i(A)\ar[r]^h & C \ar[r]^g & i(B).}
\]
The above diagram induces an isomorphism of left triangles in $\X$:
\[
\xymatrix{
Q\Omega i Qi(B)\ar[r]^{\ \ \ \ Q(m)} \ar[d]_{Q\Omega i(\eta_{_B}^{-1})}& Q(D)\ar[r]^{\eta_{_{Q(C)}}^{-1}Q(l)}\ar[d]_{Q(t)} & Q(C)\ar@{=}[d] \ar[r]^{Q(g)} & Qi(B)\ar[d]^{\eta_{_B}^{-1}}  \\
Q\Omega i(B) \ar[r]^{-Q\Omega i(f)} & Q\Omega i(A)\ar[r]^{Q(h)} & Q(C) \ar[r]^{\eta_{B}^{-1}Q(g)} & B}
\]
where $Q(t)$ is an isomorphism by assumption and the commutativity of the diagram is by noting that $\varepsilon_{_{i(B)}}=i(\eta_{_B}^{-1})$ and $\eta_{_{Q(C)}}^{-1}=Q(\varepsilon_{_C})$. Thus the second row is a left triangle in $\triangle_\X$ since the first one is.

(LT3)\ Without loss of generality, we may only consider the following diagram of left triangles in $\triangle_\X$
\[
\xymatrix{
Q\Omega i(A)\ar[r]^{Q(h)} \ar[d]_{Q\Omega i(\alpha)}& Q(C)\ar[r]^{\ \ \eta_{_{B}}^{-1}Q(g)} & B\ar[d]^{\beta} \ar[r]^{f} & A\ar[d]^{\alpha}  \\
Q\Omega i(A') \ar[r]^{Q(h')} & Q(C')\ar[r]^{\ \ \eta_{_{B'}}^{-1}Q(g')} & B' \ar[r]^{f'} & A'.}
\]
By the axiom (LT3) of $\triangle$, we have a commutative diagram of left triangles in $\triangle$:
\[
\xymatrix{
\Omega i(A)\ar[r]^{h} \ar[d]_{\Omega i(\alpha)}& C\ar[r]^{g} \ar[d]^l & i(B)\ar[d]^{i(\beta)} \ar[r]^{i(f)} & A\ar[d]^{i(\alpha)}  \\
\Omega i(A') \ar[r]^{h'} & C'\ar[r]^{g'} & i(B') \ar[r]^{i(f')} & i(A').}
\]
Then $Q(l)$ is the desired filler.

(LT4) \ We claim that given any left triangle $\Omega(A)\stackrel{h}\to C\stackrel{g}\to B\stackrel{f}\to A$ in $\triangle$, the left triangle $Q\Omega iQ(A)\stackrel{Q(h\Omega(\varepsilon_{_A}))}\To Q(C)\stackrel{Q(g)}\to Q(B)\stackrel{Q(f)}\to Q(A) $ is in $\triangle_\X$. This can be proved similarly to the proof of (LT2). Given two composable morphisms $C\stackrel{g}\to B$ and $B\stackrel{f}\to A$ in $\X$, applying the (LT4) axiom of $\triangle$ to $i(f)$ and $i(g)$ we get a commutative diagram of left triangles in $\C$:
\[
\xymatrix{
& \Omega(F)\ar[d]^{m\Omega(l)} &  & \\
\Omega i(B)\ar[r]^m\ar[d]_{\Omega(f)} & D\ar[r]^{k}\ar[d]^{\alpha} &i(C)\ar@{=}[d] \ar[r]^{i(g)} & i(B)\ar[d]^{i(f)}  \\
\Omega i(A) \ar[r]^{n}\ar@{=}[d] & E\ar[r]^{h}\ar[d]^\beta & i(C) \ar[r]^{i(fg)} \ar[d]^{i(g)}& i(A)\ar@{=}[d]\\
\Omega i(A)\ar[r]^{s} & F\ar[r]^{l}& i(B)\ar[r]^{i(f)}& i(A)}
\]
which induces the desired commutative diagram of left triangles in $\triangle_\X$
\[
\xymatrix{
& Q\Omega i Q(F)\ar[d]^{Q(m)Q\Omega i(\eta_{_B}^{-1}Q(l))} &  & \\
Q\Omega i(B)\ar[r]^{Q(m)}\ar[d]_{Q\Omega i(f)} & Q(D)\ar[r]^{\ \ \eta_{_C}^{-1}Q(k)}\ar[d]^{Q(\alpha)} &C\ar@{=}[d] \ar[r]^{g} & B\ar[d]^{f}  \\
Q\Omega i(A) \ar[r]^{Q(n)}\ar@{=}[d] & Q(E)\ar[r]^{\ \ \eta_{_C}^{-1}Q(h)}\ar[d]^{Q(\beta)} & C \ar[r]^{fg} \ar[d]^{g}& A\ar@{=}[d]\\
Q\Omega i(A)\ar[r]^{Q(s)} & Q(F)\ar[r]^{\ \ \eta_{_B}^{-1}Q(l)}& B\ar[r]^{f}& A}
\]
where the second column is a left triangle in $\triangle_\X$ is by the above claim and the equality $$l\varepsilon_{_F}=\varepsilon_{_{i(B)}}iQ(l)=i(\eta_{_B}^{-1})iQ(l).$$

$(ii)$ can be proved dually. \end{proof}

\subsection{Priles of one-sided categories}
We first recall the notation of a pretriangulated category in \cite[Definition 4.9]{Beligiannis01} and \cite[Definition 6.5.1]{Hovey99}.
\begin{definition} \ A {\it pretriangulated structure} on an additive category $\C$ is a quadruple $(\Sigma, \Omega, \triangle, \bigtriangledown)$ such that
\vskip5pt
$(i)$ \ $(\Sigma,\Omega)$ is an adjoint pair with unit $\eta$ and counit $\varepsilon$.
\vskip5pt
$(ii)$ \ $(\Omega, \triangle)$ is a left triangulated structure on $\C$ and $(\Sigma, \bigtriangledown)$ is a right triangulated structure on $\C$.
 \vskip5pt
 $(iii)$ \ For any diagram in $\C$ with commutative left square
\[
\xymatrix{
A\ar[r]^f\ar[d]_{\alpha} & B\ar[r]^g\ar[d]_{\beta} &C\ar@{.>}[d]_{\delta} \ar[r]^h & \Sigma(A)\ar[d]^{\varepsilon_{C'}\Sigma(\alpha)}    \\
\Omega(C') \ar[r]^{f'} & \Omega(A')\ar[r]^{g'} & B' \ar[r]^{h'} & C' }
\]
where the upper row is in $\bigtriangledown$ and the lower row is in $\triangle$, there exists a morphism $\delta$ making the diagram commutative.
 \vskip5pt
 $(iv)$ \ For any diagram in $\C$ with commutative right square
\[
\xymatrix{
A\ar[r]^f\ar[d]_{\Omega(\alpha)\eta_{A}} & B\ar[r]^g\ar@{.>}[d]_{\delta} &C\ar[d]_{\beta} \ar[r]^h & \Sigma(A)\ar[d]_{\alpha}    \\
\Omega(C') \ar[r]^{f'} & \Omega(A')\ar[r]^{g'} & B' \ar[r]^{h'} & C' }
\]
where the upper row is in $\bigtriangledown$ and the lower row is in $\triangle$, there exists a morphism $\delta$ making the diagram commutative.
\end{definition}

An additive category $\C$ is called a {\it pretriangulated category} if there is a pretriangulated structure on $\C$.

\begin{definition} \ Let $\C_1=(\C_1,\Omega, \triangle)$ be a left triangulated category, $\C_2=(\C_2, \Sigma, \bigtriangledown)$ a right triangulated category and $\X$ an additive category. Let $i:\X\to \C_1$ be a left coreflective functor and $j:\X\to \C_2$ a right reflective functor. The triple $(\C_1, \X, \C_2)$ is called a {\it prile of one-sided triangulated categories} if the following conditions hold:
\vskip5pt
$(i)$ \ $(Q\Omega i, R\Sigma j)$ is an adjoint pair, where $Q$ is the coreflector of $i$ and $R$ the reflector of $j$.
\vskip5pt
$(ii)$ \ Let $\triangle_{\X}$ be the class of left triangles and $\bigtriangledown_\X$ the class of right triangles of $\X$ as constructed in Proposition 2.4. For any diagram
\[
\xymatrix{
Q\Omega i (A)\ar[r]^{\ \ \ h}\ar@{=}[d] & C \ar[r]^{g'}\ar@{=}[d] &D\ar@{.>}[d]_{s} \ar[r]^{f'\ \ \ \ \ \ } & R\Sigma jQ\Omega i(A)\ar[d]^{\varepsilon_{_A}}   \\
Q \Omega i(A) \ar[r]^{\ \ \ h} & C\ar[r]^{g} & B \ar[r]^{f} & A }
\]
with the upper row is in $\bigtriangledown_\X$ and the lower row is in $\triangle_\X$ (where $\varepsilon$ is the counit of the adjunction $(Q\Omega i, R\Sigma j)$), there exists a morphism $s$ such that the diagram commutes.

\vskip5pt
$(iii)$ \ For any diagram
\[
\xymatrix{
X\ar[r]^{l}\ar[d]_{\eta_{_X}} & Y \ar[r]^{m}\ar@{.>}[d]_{t} &Z\ar@{=}[d] \ar[r]^{n\ \ \ \ } & R\Sigma j(X)\ar@{=}[d]   \\
Q \Omega i R\Sigma j(X) \ar[r]^{\ \ \ \ \ \ \ \ \ l'} & U\ar[r]^{m'} & Z \ar[r]^{n\ \ \ \ } & R\Sigma j(X) }
\]
with the upper row is in $\bigtriangledown_\X$ and the lower row is in $\triangle_\X$ (where $\eta$ is the unit of the adjunction $(Q\Omega i, R\Sigma j)$), there exists a morphism $t$ such that the diagram commutes.
\end{definition}

For example, any recollement $(\T_1,\T,\T_2)$ of triangulated categories gives rise to a prile of one-sided triangulated categories $(\T, \T_1,\T)$.

By the definition of a prile of one-sided triangulated categories and Proposition 2.4, it is not difficult to prove the following result:

\begin{theorem} \ Let $(\C_1, \X, \C_2)$ be a prile of one-sided triangulated categories. Then $(Q\Omega i, R\Sigma j, \triangle_\X, \bigtriangledown_\X)$ is a pretriangulated structure on $\X$.
\end{theorem}

\begin{remark}\ Given a prile of one-sided triangulated categories $(\C_1,\X,\C_2)$, when we say that $\X$ inherits a pretriangulated structure from $\C_1$ and $\C_2$, we always mean that the pretriangulated structure is given as in Theorem 2.7.
\end{remark}

\section{Exact model categories and their homotopy categories}
This section is aimed to review the standard material of exact model categories and their homotopy categories.

\subsection{Model categories}\ Recall that in a model category $\A$, there are three classes of morphisms, called {\it cofibrations, fibrations} and {\it weak equivalences}. We will denote them by $\C of, \mathcal{F} ib$ and $\mathcal{W} e$, respectively. A morphism which is both a weak equivalence and a (co-)fibration is called an {\it acyclic $($co-$)$fibration}.  For details about model categories, we refer the reader to \cite[Chapter I]{Quillen67}, \cite{Dwyer/Spalinski95}, \cite[Chapter 1 ]{Hovey99} and \cite[Chapter 8]{Hirschhorn03}.

Given a model category, an object $A\in \A$ is called {\it cofibrant} if $0\to A\in \mathcal{C} of$, it is called {\it fibrant} if $A\to 0\in \mathcal{F}ib$, and it is called {\it trivial} if $0\to A\in \mathcal{W}e$. We use $\A_c, \A_f, \A_{triv}$ to denote the classes of cofibrant, fibrant, and trivial objects, respectively. Each object of $\A_{cf}:=\A_c\cap \A_f$ is called {\it bifibrant}.

\subsection{Exact model categories}
The concept of an exact category is due to D. Quillen \cite{Quillen73}, a simple axiomatic description can be founded in \cite[Appendix A]{Keller90}.  Roughly speaking {\it an exact category } is an additive category $\A$ equipped with a class $\mathcal{E}$ of {\it kernel-cokernel sequences} $A\stackrel{s}\to B\stackrel{t} \to C$ in $\A$ such that $s$ is the kernel of $t$ and $t$ is the cokernel of $s$. The class $\mathcal{E}$ satisfies exact axioms, for details, we refer the reader to \cite[Definition 2.1]{Buhler10}. Given an exact category $\A$, we will call a kernel-cokernel sequence {\it short exact} if it is in $\mathcal{E}$. The map $s$ in a short exact sequence $0\to A\stackrel{s}\to B\stackrel{t} \to C\to 0$ is called an {\it inflation} and the map $t$ is called a {\it deflation}.  An exact category $\A$ is called {\it weakly idempotent complete} if every split monomorphism (a map with left inverse) is an inflation or equivalently every split epimorphism (a map with right inverse) is a deflation.
\begin{definition} (\cite[Definition 2.1]{Hovey02}) \ An {\it exact model category} is a weakly idempotent complete additive category which has both an exact structure and a model category such that the cofibrations are the inflations with cofibrant cokernels and the fibrations are the deflations with fibrant kernels.
\end{definition}

\begin{remark} \
For an exact model category $\A$, the acyclic cofibrations are precisely the inflations with trivial cofibrant cokernels, the acyclic fibrations are precisely the deflations with trivial fibrant kernels; see \cite[Proposition 4.2]{Hovey02}.
\end{remark}

\begin{definition} \ \cite[Definition 2.1]{Gillespie11} Let $\A$ be an exact category. A {\it cotorsion pair} in $\A$ is a pair $(\mathcal{D}, \mathcal{E})$ of classes of objects of $\A$ such that
\vskip5pt
$(i)$ \ $\mathcal{D}=\{D\in \A \ | \ {\rm Ext}^1_\A(D, \mathcal{E})=0\}.$
\vskip5pt
$(ii)$ \ $\mathcal{E}=\{E\in \A \ | \ {\rm Ext}^1_\A(\mathcal{D}, E)=0\}.$
\end{definition}
The cotorsion pair $(\mathcal{D}, \mathcal{E})$ is called {\it complete} if the following conditions are satisfied:
\vskip5pt
$(iii)$ \ $(\mathcal{D}, \mathcal{E})$ has {\it enough projectives}, i.e. for each $A\in \A$ there exists a short exact sequence $0\to E\to D\to A\to 0$ such that $D\in \mathcal{D}$ and $E\in \mathcal{E}$.
\vskip5pt
$(iv)$ \ $(\mathcal{D}, \mathcal{E})$ has {\it enough injectives}, i.e. for each $A\in \A$ there exists a short exact sequence $0\to A\to E\to D\to 0$ such that $D\in \mathcal{D}$ and $E\in \mathcal{E}$.
\vskip5pt

For a convenient statement, we follow Gillespie \cite[Definition 3.1]{Gillespie12} to define the notion of a {\it Hovey triple}. Recall that a full subcategory $\mathcal{W}$ of an exact category $\A$ is called {\it thick} if it is closed under direct summands and if two out of three of the terms in a short exact sequence are in $\mathcal{W}$, then so is the third. A triple $(\mathcal{D}, \mathcal{W}, \mathcal{E})$ of classes of objects in an exact category $\A$ is called a {\it Hovey triple} if $\mathcal{W}$ is a thick subcategory of $\A$ and both $(\mathcal{D}, \mathcal{W}\cap \mathcal{E})$ and $(\mathcal{D}\cap \mathcal{W}, \mathcal{E})$ are complete cotorsion pairs in $\A$.

The fundamental result of exact model categories is the following Hovey's correspondence:
\begin{theorem} \ $($\cite[Theorem 2.2]{Hovey02}$)$ \ Let $\A$ be an exact category. There is a bijection:
 \[
\left\{\begin{gathered}\text{exact model structures  }\\ \text{on  $\A$}
\end{gathered}\;
\right\} \xymatrix@C=1pc{ \ar[r]^-{\sim} & }\left\{
\begin{gathered}
  \text{Hovey triples in $\A$}
\end{gathered}
\right\}\,
\]
which sends an exact model structure on $\A$ to $(\A_c, \A_{triv}, \A_f)$.
\end{theorem}
Here the statement in this generality is stated in \cite[Theorem 3.3]{Gillespie11}.
\begin{remark}
If $\A$ is an abelian category which has enough projective and injective objects, the above correspondence is also discussed in \cite[Chapter VIII]{Beligiannis/Reiten07}.
\end{remark}
\subsection{Stable categories }\ Recall that an additive subcategory $\X$ of an additive category $\C$ is called {\it contravariantly finite} in $\C$ if each object $C$ of $\C$ has a {\it right $\X$-approximation}, i.e. there is a morphism $X_C\to C$ with $X_C\in \X$ such that the induced map $\Hom_\C(X, X_C)\to \Hom_\C(X, C)$ is surjective for all $X\in \X$. Dually,
 $\X$ is called {covariantly finite} in $\C$ if every object $C$ has a {\it left $\X$-approximation}, i.e. there is a morphism $C\to X^C$ with $X^C\in \X$ such that the induced map $\Hom_\C(C, X)\to \Hom_\C(X^C, X)$ is injective for all $X\in \X$.

 Given two morphisms $f,g: C\to D$ in $\C$, we say that $f$ is {\it stably equivalent } to $g$ with respect to $\X$, written $f\stackrel{\X}\sim g$, if $f-g$ factors through some object of $\X$. It is well known that stable equivalence is an equivalence relation which is compatible with compositions. That is, if $f\stackrel{\X}\sim g$, then $fk\stackrel{\X}\sim gk$ and $hf\stackrel{\X}\sim hg$ whenever the compositions make sense. Let $\C/\X$ be the stable category. That is the category whose objects are the same as $\C$ and whose morphisms are the stable equivalence classes of $\C$ with respect to $\X$. Let $\pi: \C\to \C/\X$ be the canonical quotient functor. The image $\pi(C)$ of $C\in \C$ is denoted by $\underline{C}$ and the image $\pi(f)$ of any morphism $f$ is denoted by $\underline{f}$.
 \vskip5pt
 For simplicity, we will suppress the underline on the objects in stable categories but preserve the underline on the morphisms. If a morphism has a subindex such as $f_g$ in $\C$, then we will write $\ul{f}_g$ instead of $\ul{f_g}$ in the stable category.

\subsection{Homotopy categories of model categories}
Given a model category $\A$, recall that a {\it path object} for an object $A\in \A$ is an object $A^I$ of $\A$ together with a factorization of the diagonal map $A\stackrel{(1_A, 1_A)}\To A\prod A $:
$A\stackrel{s}\to A^I\stackrel{(p_0,p_1)}\To A\prod A $
 where $s$ is a weak equivalence and $(p_0,p_1)$ a fibration such that $p_0s=p_1s=1_A$. A path object is called {\it very good} if in addition $s$ is a cofibration.

 Two morphisms $f,g:A\to B$ in $\A$ are called {\it right homotopic} if there exists a path object $B^I$ for $B$ and a morphism $H: A\to B^I$ such that $f=p_0H$ and $g=p_1H$. In this case, $H$ is called a {\it right homotopy} from $f$ to $g$. If $f$ and $g$ are right homotopic, we denote this by $f\stackrel{r}\sim g$. Homotopic morphisms become equal in the homotopy category. The relation $\stackrel{r}\sim$ is compatible with composition in $\A_c$; see \cite[Lemma I.1.6]{Quillen67} and \cite[Lemma 4.17]{Dwyer/Spalinski95}. We denote by $\A_c/{\stackrel{r}\sim}$ the quotient category of $\A_c$ with respective to the equivalence relation generated by $\stackrel{r}\sim$.

 Dually, one can define the notions of {\it cylinder objects}, {\it left homotopic} $\stackrel{l}\sim$, {\it left homotopies} and the quotient category $\A_f/{\stackrel{l}\sim}$, respectively.

 In the subcategory of bifibrant objects $\A_{cf}$, the relations $\stackrel{r}\sim$ and $\stackrel{l}\sim$ coincide and yield an equivalence relation. We denote this relation by $\stackrel{h}\sim$. By \cite[Theorem VII.4.2]{Beligiannis/Reiten07}, \cite[Proposition 4.3,4.7]{Gillespie11}, \cite[Proposition 1.1.14]{Becker12}, if $\A$ is an exact model category, the homotopy relation on $\A_{cf}$ is the same as the stably equivalence relation $\stackrel{\omega}\sim$ (see Subsection 3.3). So the corresponding quotient category $\A_{cf}/\stackrel{h}\sim=\A_{cf}/\omega$.

 The {homotopy category} of $\A$ is the Gabriel-Zisman localization \cite{Gabriel/Zisman67} of $\A$ with respective to the class of weak equivalences and is denoted by $\gamma: \A\to \Ho(\A)$. We will use $\gamma_c:\A_c\to \Ho(\A_c)$ (resp. $\gamma_f:\A_f\to \Ho(\A_f)$) to denote the localization of $\A_c$ (resp. $\A_f$) with respect to the class of maps in $\A_c$  (resp. $\A_f$) which are weak equivalences in $\A$.

 D. Quillen's homotopy category theorem \cite[Theorem I.1.1]{Quillen67} gives the following commutative diagram:

\[
\xymatrix{
\A_f/{\stackrel{l}\sim} \ar[d]_{\bar{\gamma}_f}& \A_{cf}/\omega \ar@{^{(}->}[r]\ar@{_{(}->}[l]\ar@{_{(}->}[d]_\sim^{\bar{\gamma}} & \A_c/{\stackrel{r}\sim} \ar[d]_{\bar{\gamma}_c}\\
\Ho(\A_f) \ar@{^{(}->}[r]^\sim & \Ho(\A) & \Ho(\A_c)\ar@{_{(}->}[l]_\sim}
\]
where $\hookrightarrow$ denotes a full embedding and $\stackrel{\sim}\to$ denotes an equivalence of categories.

D. Quillen also has shown that there was extra structure on $\Ho(\A)$ such as the suspension and loop functors and the families of fibration and cofibration sequences. The homotopy category category together with this structure is called the {\it homotopy theory } of the model category $\A$; see \cite{Quillen69}.
When $\A$ is in addition an additive category, it just means that with this structure $\Ho(\A)$ is a pretriangulated category in the sense of Definition 2.6.
We refer the reader to \cite[Chapter I]{Quillen67} for explicit construction of the homotopy theory of a model category.
\section{Priles of one-sided triangulated categories arising from exact model categories}

In this section, we fix an exact model category $\A$.
\subsection{} By Hovey's correspondence, there are two complete cotorsion pairs $(\A_c, \A_{triv}\cap \A_f)$ and $(\A_c\cap \A_{triv}, \A_f)$. Let $\omega=\A_{cf}\cap \A_{triv}$. Then $\omega$ is contravariantly finite in $\A_f$ and covariantly finite in $\A_c$. So for each object $A$ in $\A_f$, we can assign a short exact sequence $0\to \ker p_{_A}\stackrel{\iota_A}\to W_A\stackrel{p_{_A}}\to A\to 0$ in $\A_f$ such that $p_{_A}$ is a right $\omega$-approximation. Given a morphism $f: B\to A$ in $\A_f$, there is a commutative diagram with exact rows:
\[
\xymatrix{
0\ar[r] &\ker p_{_B}\ar[r]^{\iota_{_B}}\ar[d]_{\kappa_f} &W_B\ar[d]_{x_f} \ar[r]^{p_{_B}} & B\ar[d]_f \ar[r] & 0 \\
0 \ar[r] & \ker p_{_A}\ar[r]^{\iota_{_A}} & W_A \ar[r]^{p_{_A}} & A\ar[r]& 0}
\]
Let $\Omega^f: \A_f/\omega\to \A_f/\omega$ be the functor which sends each object $A\in \A_f$ to $\ker p_{_A}$ and each morphism $\underline{f}$ to $\ul{\kappa}_f$ ( this functor is well-defined, for details; see \cite[Section 1]{Beligiannis/Marmaridis94} or \cite[Subsection 2.2]{ZWLi}). Let $\triangle_f$ be the class of left triangles which are isomorphic to those of the form $\Omega^f(A)\stackrel{\ul{\zeta}_f}\to {\rm PB}(f)\stackrel{\ul{\mu}_f}\to B\stackrel{\ul{f}}\to A $ in the stable category $\A_f/\omega$. Where the left triangle fits into the following commutative diagram in $\A_f$
\[
\xymatrix{
0  \ar[r] & \ker p_{_X} \ar[r]^{\zeta_{_f} \ \ }\ar@{=}[d] & {\rm PB}(f)\ar[d]_{\theta_{_f}} \ar[r]^{~~~\ \mu_{_f}} & B \ar[d]_f ^{\ \ \ \ \ \ \ \ \ \qquad(*)}\ar[r]& 0\\
0  \ar[r] & \ker p_{_X} \ar[r]^{\iota_A} & W_A \ar[r]^{p_{_{A}}} & A \ar[r] & 0}
\]
with exact rows and the right square being a pullback.
\vskip5pt
Dually, one can construct an endofunctor $\Sigma^c$ of $\A_c/\omega$ and a class of right triangles $\bigtriangledown_c$ in $\A_c/\omega$.
\vskip5pt
Similar to Theorem 4.6 of \cite{ZWLi} in the case of abelian model categories, we have the following lemma:
\begin{lemma}  \ Let $\A$ be an exact model category. With the above notations, we have
\vskip5pt
$(i)$ \ $(\A_f/\omega, \Omega^f, \triangle_f)$ is a left triangulated category.
\vskip5pt
$(ii)$ \ $(\A_c/\omega, \Sigma^c, \bigtriangledown_c)$ is a right triangulated category.
\end{lemma}

\subsection{Priles of one-sided triangulated categories arising from exact model categories}

 Since the subcategory of bifibrant objects $\A_{cf}$ is contravariantly finite in $\A_f$ and the cotorsion pair $(\A_{c}, \A_{tri}\cap \A_f)$ is complete,  we can assign each object $A\in\A_f$ a short exact sequence $$0\to \ker q_{_A}\to Q(A)\stackrel{q_{_A}}\to A\to 0$$
 such that $q_{_A}$ is a right $\A_{cf}$-approximation and $\ker q_{_A}\in \A_{tri}\cap \A_f$. In particular, $q_{_A}$ is a weak equivalence. If $A\in \A_{cf}$ we just let $Q(A)=A$.
 Given a morphism $f: B\to A$ in $\A_f$, there is a commutative diagram with exact rows:
\[
\xymatrix{
0\ar[r] & \ker q_B\ar[r]^{\epsilon_{_B}}\ar[d] &Q(B)\ar[d]_{y_f} \ar[r]^{q_{_B}} & B\ar[d]_f \ar[r]& 0 \\
0 \ar[r] & \ker q_A\ar[r]^{\epsilon_{_A}} & Q(A) \ar[r]^{q_{_A}} & A \ar[r]& 0.}
\]

 Dually, $\A_{cf}$ is covariantly finite in $\A_c$ and we can assign each object $X\in \A_c$ a short exact sequence $$0\to X \stackrel{r^X}\to R(X)\to \cok r^X\to 0$$
 such that $r^X$ is a left $\A_{cf}$-approximation and $\cok r^X\in \A_{tri}\cap \A_c$ (so $r^X$ is a weak equivalence). If $X\in \A_{cf}$ we just let $R(X)=X$. Given a morphism $g: X\to Y$ in $\A_c$, there is a commutative diagram with exact rows:
\[
\xymatrix{
0\ar[r] & X\ar[r]^{r^X}\ar[d]_{g} &R(X)\ar[d]_{a^g} \ar[r] & \cok r^X\ar[d] \ar[r]& 0 \\
0 \ar[r] & Y\ar[r]^{r^Y} & R(Y) \ar[r] & \cok r^Y \ar[r]& 0.}
\]
Combined with the notations ${\rm PB}(f)$ and ${\rm PO}(g)$ as defined in $(*)$ and its dual, we have:
\begin{lemma} \
$(i)$ For any morphism $f:B\to A$ in $\A_f$, the induced morphism ${\rm PB}(y_{_f})\to {\rm PB}(f)$ is a weak equivalence.
\vskip5pt
$(ii)$ For any morphism $g:X\to Y$ in $\A_c$, the induced morphism ${\rm PO}(g)\to {\rm PO}(a^g)$ is a weak equivalence.
\end{lemma}
\begin{proof} \ We only prove assertion $(ii)$ since $(ii)$ is the dual case. Applying the Gluing Lemma \cite[Lemma 1.4.1 (2)]{Radulescu-Banu06} to the commutative diagram
\[
\xymatrix{
 Q(A)\oplus W_{Q(A)}\ar[r]^{\left(\begin{smallmatrix}
1 & p_{_{Q(A)}} \\
1 & 0
\end{smallmatrix}\right)}\ar[d]_{\left(\begin{smallmatrix}
q_{_A} & 0 \\
0 & x_{q_{_A}}
\end{smallmatrix}\right)} &Q(A)\oplus Q(A)\ar[d]_{\left(\begin{smallmatrix}
q_{_A} & 0 \\
0 & q_{_A}
\end{smallmatrix}\right)} & Q(B) \ar[l]_{\ \ \ \ \ \ \left(\begin{smallmatrix}
y_{_f} \\
0
\end{smallmatrix}\right)} \ar[d]^{q_{_B}}  \\
 A\oplus W_A\ar[r]_{\left(\begin{smallmatrix}
1 & p_{_A} \\
1 & 0
\end{smallmatrix}\right)} & A\oplus A & B\ar[l]_{\left(\begin{smallmatrix}
f \\
0
\end{smallmatrix}\right)}}
\]
we know that the induced morphism ${\rm PB}(y_{_f}) \to {\rm PB}(f)$ is a weak equivalence. Where the first column is a weak equivalence is by the 2-out-of-3 axiom of weak equivalences since we have an equality:
$$\left(\begin{smallmatrix}
q_{_A} & 0 \\
0 & x_{q_{_A}}
\end{smallmatrix}\right)\left(\begin{smallmatrix}
1 \\
0
\end{smallmatrix}\right)=\left(\begin{smallmatrix}
1  \\
0
\end{smallmatrix}\right)q_{_A}:Q(A)\to A\oplus W_A.$$
Note that the morphisms $\left(\begin{smallmatrix}
1  \\
0
\end{smallmatrix}\right):Q(A)\to Q(A)\oplus W_{Q(A)}, \left(\begin{smallmatrix}
1  \\
0
\end{smallmatrix}\right):A\to A\oplus W_A, q_{_A}$ are weak equivalences by the construction of exact model structure of $\A$.
\end{proof}

We will use $i:\A_{cf}/\omega\hookrightarrow \A_f/\omega$ and $j:\A_{cf}/\omega\hookrightarrow \A_c/\omega$ to denote the inclusion functors respectively.

 \begin{lemma} \
$(i)$ \ $\A_{cf}/\omega$ is a coreflective subcategory of $\A_f/\omega$.
\vskip5pt
$(ii)$ \ $\A_{cf}/\omega$ is a reflective subcategory of $\A_c/\omega$.
\end{lemma}

\begin{proof} \ For $(i)$, by Proposition 4.4 of \cite{Gillespie11} and Lemma I.1.7 of \cite{Quillen67}, we know that if $\ul{f}=\ul{g}$ in $\A_f/\omega$ then $\ul{y}_f=\ul{y}_g$ in $\A_{cf}/\omega$. So we can define a functor $Q:\A_f/\omega \to \A_{cf}/\omega$ by sending $A$ to $Q(A)$ and $\ul{f}:B\to A$ to $\ul{y}_f$. By $(5)$ of Proposition 4.4 of \cite{Gillespie11} and Lemma I.1.7 of \cite{Quillen67} again, for $B\in \A_{cf}$ and $A\in \A_f$, there is a natural isomorphism:
$$q_{_{A,*}}=\Hom_{\A_f/\omega}(B, q_{_A}):\Hom_{\A_{cf}/\omega}(B,Q(A))\stackrel{\cong}\to \Hom_{\A_f/\omega}(i(B), A)$$
induced by the right $\A_{cf}$-approximation $q_{_A}:Q(A)\to A$. Thus $Q$ is a right adjoint of the inclusion functor $i:A_{cf}/\omega\to \A_{f}/\omega$.

For $(ii)$, dually, we can define a functor $R:\A_c/\omega\to \A_{cf}/\omega$ by sending $X$ to $R(X)$ and $\ul{f}:X\to Y$ to $\ul{a}^f$. The functor $R$ is a left adjoint of the inclusion functor $j:\A_{cf}/\omega\to \A_c/\omega$ and the adjunction isomorphism is induced by the left $\A_{cf}$-approximation $r^X:X\to R(X)$ for any $X\in \A_c$.
\end{proof}
\begin{remark} $(1)$\ Note that for $A\in \A_{cf}$ we assume that $Q(A)=A$, so the unit of the adjunction $(i,Q)$ is the identity. Dually, the counit of the adjunction in $(R,j)$ is the identity.

$(2)$\ For simplicity, we will omit the notations of $i$ and $j$ on objects.
\end{remark}
\begin{lemma} \ $(R\Sigma^ci, Q\Omega^fj )$ is an adjoint pair on $\A_{cf}/\omega$.
\end{lemma}
\begin{proof} \ Firstly, we prove that there is a natural isomorphism
$$\Hom_{\A_c/\omega}(\Sigma^c(A), B)\cong \Hom_{\A_f/\omega}(A, \Omega^f(B))$$
on $A, B\in \A_{cf}$. For each $\ul{\alpha}\in \Hom_{\A_{c}/\omega}(\Sigma^c(A),B)$, by the constructions of $\Sigma^c(A)$ and $\Omega^f(B)$, there is a commutative diagram
\[
\xymatrix{
0\ar[r] & A\ar[r]^{\tau^{A}}\ar[d]_{\lambda_\alpha} & W^{A}\ar[d] \ar[r]^{\pi^{A}} & \Sigma^cA\ar[d]_\alpha \ar[r]& 0   \\
0 \ar[r] & \Omega^fB\ar[r]^{\iota_{_{B}}} & W_{B} \ar[r]^{p_{_{B}}} & B \ar[r] &0. }
\]
The middle morphism exists since $p_{_{B}}$ is a right $\omega$-approximation and $W^{A}\in \omega$. It can be verified directly that $\ul{\lambda}_\alpha$ is uniquely decided by the stable equivalence classes of $\alpha$ in $\A_f/\omega$. Dually, each $\ul{\beta}\in \Hom_{\A_f/\omega}(A,\Omega^f(B))$, determines uniquely a morphism $\ul{\lambda}^\beta: \Sigma^c(A)\to B $ in $\A_c/\omega$
\[
\xymatrix{
0\ar[r] & A\ar[r]\ar[d]_{\beta} & W^A\ar[d]_{z^{\beta}} \ar[r]^{\pi^{A}} & \Sigma^c(A)\ar[d]_{\lambda^\beta} \ar[r]& 0   \\
0 \ar[r] & \Omega^f(B)\ar[r] & W_B \ar[r] & B \ar[r] &0. }
\]
So we can define a map
$$\varphi_{A,B}: \Hom_{\A_c/\omega}(\Sigma^c(A), B)\to \Hom_{\A_f/\omega}(A, \Omega^f(B))$$
 by sending $\ul{\alpha}$ to $-\ul{\lambda}_{\alpha}$, it has an inverse which sends $\ul{\beta}$ to $-\ul{\lambda}^{\beta}$. The naturalness of $\varphi$ can be verified directly.

Denote by $r^{\Sigma^c A,*}=\Hom_{\A_c/\omega}(r^{\Sigma^c (A)}, B)$ and $q_{_{\Omega^fB,*}}=\Hom_{\A_f/\omega}(A, q_{_{\Omega^fB}})$. Then the composition $$q_{\Omega^f(B),*}^{-1}  \varphi_{A,B}  r^{\Sigma^c (A), *}: \Hom_{\A_{cf}/\omega}(R\Sigma^c(A), B)\to \Hom_{\A_{cf}/\omega}(A,Q\Omega^f(B))$$ is the desired adjunction isomorphism. \end{proof}

\begin{proposition}\ Let $\A$ be an exact model category.
\vskip5pt
$(i)$ \ $\A_{cf}/\omega$ is a left coreflective subcategory of $(\A_f/\omega, \Omega^f, \triangle_f)$.
\vskip5pt
$(ii)$ \ $\A_{cf}/\omega$ is a right reflective subcategory of $(\A_c/\omega, \Sigma^c, \bigtriangledown_c)$.
\end{proposition}

\begin{proof} \ We only prove $(i)$ and leave the dual case to the reader.

$(i)$. \ By Lemma 4.3, $A_{cf}/\omega$ is a coreflective subcategory of $\A_f/\omega$ with coreflector $Q:\A_f/\omega\to \A_{cf}/\omega$. So we only need to check condition $(ii)$ of Definition 2.3. Given any morphism $f:B\to A$ in $\A_f$, by Lemma 4.1, there is a commutative diagram of distinguished left triangles in $\A_f/\omega$:
\[
\xymatrix{
\Omega^f  Q(A)\ar[r]^{\ \ \ \ul{\zeta}_{y_{_f}}} \ar[d]_{\Omega^f(\ul{q}_{A})}& {\rm PB}(y_{_f})\ar[r]^{\ul{\mu}_{y_{_f}}}\ar[d]_{\ul{t}} & Q(B)\ar[d]_{\ul{q}_{B}} \ar[r]^{\ul{y}_{f}} & Q(A)\ar[d]^{\ul{q}_{A}}  \\
\Omega^f(A) \ar[r]^{\ul{\zeta}_f} & {\rm PB}(f)\ar[r]^{\ul{\mu}_f} & B \ar[r]^{\ul{f}} & A.}
\]
where $t$ is the induced morphism by the universal property of pullbacks. By Lemma 4.2, $t$ is a weak equivalence and thus $Q(\ul{t})$ is an isomorphism in $\A_{cf}/\omega$.
\end{proof}
We will use $(Q\Omega^f i, \triangle_{cf})$ to denote the induced left triangulated structure on $\A_{cf}/\omega$ from $\A_f/\omega$ and use $(R\Sigma^cj, \bigtriangledown_{cf})$ to denote the induced right triangulated structure on $\A_{cf}/\omega$ from $\A_c/\omega$.

\begin{theorem} \ Let $\A$ be an exact model category. Then $(\A_f/\omega, \A_{cf}/\omega,
 \A_c/\omega)$ is a prile of one-sided triangulated categories.
\end{theorem}
\begin{proof}
By Lemma 4.1, $(\A_f/\omega, \Omega^f, \triangle_f)$ is a left triangulated category and $(\A_c/\omega, \Sigma^c, \bigtriangledown_c)$ is a right triangulated category. By Proposition 4.6, $\A_{cf}/\omega$ is a left coreflective subcategory of $\A_f/\omega$ with induced left triangulated structure $(Q\Omega^f i, \triangle_{cf})$ and $\A_{cf}/\omega$ is a right reflective subcategory of $\A_c/\omega$ with induced right triangulated structure $(R\Sigma^c j, \bigtriangledown_{cf})$. By Lemma 4.5, $(R\Sigma^c j, Q\Omega^f i)$ is an adjoint pair. So we only need to check conditions $(ii)$ and $(iii)$ of Definition 2.6.

From the proof of Lemma 4.5, we know that the counit $\varepsilon$ of the adjunction $(R\Sigma^c j, Q\Omega^f i)$ is defined by the equality $$\varepsilon_{_A} r^{\Sigma^c Q\Omega^f(A)}=-\lambda^{q_{_{\Omega^f(A)}}}$$
for any $A\in \A_{cf}$.

For condition $(ii)$ of Definition 2.6, without loss of generality, we may only consider the following diagram (see Proposition 2.4, Lemma 4.1 and Remark 4.4):
\[
\xymatrix{
Q\Omega^f(A)\ar[r]^{\ \ \ul{y}_{{\zeta_{_f}}}\ \ }\ar@{=}[d] & Q({\rm PB}(f)) \ar[r]^{\ul{r}^{{\rm PO}(y_{\zeta_{_f}})}\ul{v}}\ar@{=}[d] & \ \ R({\rm PO}(y_{\zeta_{_f}})) \ar[r]^{\ul{a}^{w}} & R\Sigma Q\Omega (A)\ar[d]^{\ul{\varepsilon}_{A}}   \\
Q \Omega^f(A) \ar[r]^{\ \  \ul{y}_{{\zeta_{_f}}}\ \ } & Q({\rm PB}(f))\ar[r]^{\ \ \ \ \ \ul{\mu}_f \ul{q}_{{\rm PB}(f)}} & B \ar[r]^{\ul{f}} & A }
\]
for a morphism $f:A\to B $ in $\A_{cf}$, where $y_{\zeta_{_f}}, v, w$ fit into a pushout diagram
\[
\xymatrix{
0  \ar[r] & Q\Omega^f(A) \ar[r]^{\ \ \tau^{Q\Omega^f(A)} \ \ }\ar[d]_{y_{\zeta_{_f}}} & W^{Q\Omega^f(A)}\ar[d]_{u} \ar[r]^{ \pi^{Q\Omega^f(A)}} & \Sigma^c Q\Omega^f(A) \ar@{=}[d] \ar[r]& 0\\
0  \ar[r] & Q({\rm PB(f)}) \ar[r]^{v} & {\rm PO}(y_{\zeta_{_f}}) \ar[r]^{w} & \Sigma^c Q\Omega^f(A) \ar[r] & 0}
\]
with the first row the assigned left $\omega$-approximation of $Q\Omega^f(A)$.

Since $$\mu_{_f}q_{_{{\rm PB}(f)}}y_{\zeta_{_f}}=\mu_{_f}\zeta_f q_{_{\Omega^f(A)}}=0$$
there is a morphism $n: {\rm PO}(y_{\zeta_{_f}})\to B$ such that
$$nv=\mu_{_f}q_{_{{\rm PB}(f)}}\ \mbox{and} \ nu=0$$
by the universal property of pushouts. Similarly, there is a morphism $m:{\rm PO}(y_{\zeta_{_f}})\to W_A$ such that $mv=\theta_{f}q_{_{{\rm PB }(f)}}$ and $mu=z^{q_{_{\Omega^f(A)}}}$ (for the notation $z^{q_{_{\Omega^f(A)}}}$, see the proof of Lemma 4.5).

Now since $$(fn- p_{_A} m+\lambda^{q_{_{\Omega^f(A)}}}w)v=0=(fn- p_{_A} m+\lambda^{q_{_{\Omega^f(A)}}}w)u$$ we have $$fn- p_{_A} m+\lambda^{q_{_{\Omega^f(A)}}}w=0$$ by the universal property of pushouts. So $$\ul{f}\ul{n}=-\ul{\lambda}^{q_{_{\Omega^f(A)}}\alpha}\ul{w}=\ul{\varepsilon}_A \ul{r}^{\Sigma^c Q\Omega^f(A)} \ul{w}= \ul{\varepsilon}_A\ \ul{a}^w \ul{r}^{{\rm PO}(y_{\zeta_{_f}})}$$
in $\A_{c}/\omega$.

By the proof of $(ii)$ of Lemma 4.3, the map $\Hom_{\A_c/\omega}(r^{{\rm PO}(y_{\zeta_{_f}})}, B)$ is an isomorphism. So there is a morphism $s: R{\rm PO}(y_{\zeta_{_f}})\to B$ such that $\ul{s} \ \ul{r}^{{\rm PO}(y_{\zeta_{_f}})}= \ul{n}$ in $\A_{c}/\omega$ and thus $\ul{f} \ \ul{s}=\ul{\varepsilon}_{A} \ul{a}^{w}$. Then $\ul{s}$ is the desired filler.

Dually, one can verify condition $(iii)$ of Definition 2.6. \end{proof}

By Theorem 2.7, we have:
\begin{corollary} \ Let $\A$ be an exact model category. Then $\A_{cf}/\omega $ inherits a pretriangulated structure from $\A_f/\omega$ and $\A_c/\omega$.
\end{corollary}

\subsection{One-sided exact functors}
By Proposition 4.4 of \cite{Gillespie11}, if $\ul{f}=\ul{g}$ in $\A_f/\omega$, then $f\stackrel{l}\sim g$, and if $\ul{f}=\ul{g}$ in $\A_c/\omega$, then $f\stackrel{r}\sim g$. So, there is an induced functor $I_f: \A_f/\omega\to \A_f/{\stackrel{l}\sim}$ which is the identity on objects and sends $\ul{f}$ to $[f]_l$, the set of left homotopy equivalence classes of $f$. Dually, there is an induced functor $I_c:\A_c/\omega \to \A_c/{\stackrel{r}\sim}$.

 Recall that a functor $F: \C_1\to \C_2$ between left triangulated categories $(\C_1, \Omega_1, \triangle_1)$ and $(\C_2, \Omega_2, \triangle_2)$ {\it left exact} if there is a natural isomorphism $\phi: \Omega_2F\stackrel{\cong}\to F \Omega_1$ such that for any distinguished left triangle $\Omega_1(A)\stackrel{h}\to C\stackrel{g}\to B\stackrel{f}\to A$ in $\C_1$, the diagram $\Omega_2F(A)\stackrel{F(h)\phi_A}\To F(C)\stackrel{F(g)}\to F(B)\stackrel{F(f)}\to F(A)$ is a distinguished left triangle in $\C_2$. Dually, one defines {\it right exact} functors between right triangulated categories.

\begin{proposition} \ $(i)$ \ The composition $\A_f/\omega\stackrel{I_f}\to \A_f/{\stackrel{l}\sim}\stackrel{\bar{\gamma}_f}\to \Ho(\A_f)$ is a left exact functor between left triangulated categories.
\vskip5pt
$(ii)$ \ The composition $\A_c/\omega\stackrel{I_c}\to \A_c/{\stackrel{r}\sim}\stackrel{\bar{\gamma}_c}\to \Ho(\A_c)$ is a right exact functor between right triangulated categories.
\end{proposition}

\begin{proof} \ We only prove $(i)$ since the statement $(ii)$ can be proved dually.

$(i)$ \ We denote by $\Omega_{\Ho(\A_f)}$ the loop functor of $\Ho(\A_f)$. We recall the left triangulated structure of $\Ho(\A_f)$ as constructed in \cite[Chapter I, Sections 2-3]{Quillen67}. For every object $A\in \A_f$, $A\oplus W_A$ is a very good path object for $A$: $$A\stackrel{\left(\begin{smallmatrix}
1 \\
0
\end{smallmatrix}\right)}\to A\oplus W_A\stackrel{\left(\begin{smallmatrix}
1 & p_{_A}\\
1 & 0
\end{smallmatrix}\right)}\to A\oplus A \qquad (**).$$
Thus $$\Omega_{\Ho(\A_f)}(A)=\ker\left(\begin{smallmatrix}
1 & p_{_A}\\
1 & 0
\end{smallmatrix}\right)=\ker p_{_A}=\Omega^f(A)$$ and for each morphism $f:B\to A$, $\Omega_{\Ho(\A_f)}(f)=\gamma_f(\kappa_f)$. Given any fibration $p:B\to A$ in $\A_f$, we have a commutative diagram with exact rows in $\A:$
\[
\xymatrix{
0 \ar[r]& \ker p_{_A}\ar[r]^{\iota_{_A}}\ar[d]_{\xi_p} & W_A\ar[d]_{\delta_g} \ar[r]^{p_{_{A}}} & A\ar@{=}[d]\ar[r]& 0  \\
0 \ar[r] & \ker p \ar[r]^{\iota_p} & B \ar[r]^{p} & A \ar[r]& 0.}
\]
Since $\A$ is an additive category, along the construction of the group action in Chapter I, Section 1.3, Proposition 1 of \cite{Quillen67}, we know that the group action of $\Omega_{\Ho(\A_f)}(A)$ on $\ker p$ is $\left(\begin{smallmatrix}
1 \\
-\xi_p
\end{smallmatrix}\right): \ker p\oplus \Omega_{\Ho(\A_f)}\to \ker p$, where $\xi_p$ is a morphism as defined in the above diagram. For, in our case, the localization $\gamma_f$ is additive and so we can always take the very good path objects for $A$ and $B$ as the form of $(**)$. For more details, see \cite[Proposition I.1.3]{Quillen67}, \cite[Proposition I.3]{Brown73} and \cite[Theorem 6.2.1, Remark 7.1.3]{Hovey99}.

So the distinguished left triangles in $\Ho(\A_f)$ are those which are isomorphic to one of the form $$\Omega_{\Ho(\A_f)}(A)\stackrel{-\gamma_f(\xi_p)}\to \ker p\stackrel{\gamma_f(\iota_p)}\to B\stackrel{\gamma_f(p)}\to A$$
where $p$ is a fibration with a kernel $\iota_p$. These distinguished left triangles together with the loop functor $\Omega_{\Ho(\A_f)}$ is the left triangulated structure on $\Ho(\A_f)$ as constructed in \cite[Proposition I.1.3.5]{Quillen67} and \cite[Subsection 6.3]{Hovey99}.

As the functor $\bar{\gamma}_fI_f$ is the identity on objects, by Proposition 2.10 of \cite{Beligiannis/Marmaridis94}, we only need to show that for any induced left triangle $\Omega^f(A)\stackrel{-\ul{\xi}_f}\to \ker f \stackrel{\ul{\iota}_f}\to B\stackrel{\ul{f}}\to A$ in $\A_f/\omega$,
the diagram $\Omega_{\Ho(\A_f)}(A)\stackrel{-\gamma_f(\xi_f)}\to \ker f\stackrel{\gamma_f(\iota_f)}\to B\stackrel{\gamma_f(f)}\to A$ is a distinguished triangle in $\Ho(\A_f)$.

The morphism $(f, p_{_A}): B\oplus W_A\to A$ is a fibration which has a kernel ${\rm PB}(f)$ as constructed in the commutative diagram $(**)$. We have a commutative diagram:
\[\xymatrix{\Omega^f(A) \ar[r]^{-\xi_f} \ar@{=}[d] & \ker f\ar[r]^{\iota_f}\ar[d]^u & B\ar[r]^f \ar[d]^{\left(\begin{smallmatrix}
1 \\
0
\end{smallmatrix}\right)} & A \ar@{=}[d]\\
\Omega^f(A)\ar[r]^{-u\xi_f}& {\rm PB}(f)\ar[r]^{\left(\begin{smallmatrix}
\mu_f \\
-\theta_f
\end{smallmatrix}\right)}& B\oplus W_A \ar[r]^{\ \ \ \ (f, p_{_A})}& A}
\]
where $u:\ker f\to {\rm PB}(f)$ is the kernel of $\theta_f$ satisfying $\mu_f u=\iota_f$ and the second row is the distinguished left triangle of $(f, p_{_A})$. The morphism $u$ is an isomorphism in $\A_f/\omega$ by the proof of Proposition 2.10 of \cite{Beligiannis/Marmaridis94}, so it is also an isomorphism in $\Ho(\A_f)$ by Proposition 4.4 of \cite{Gillespie11} . Since $\left(\begin{smallmatrix}
1 \\
0
\end{smallmatrix}\right)$ is a weak equivalence, it is an isomorphism in $\Ho(\A_f)$. So the diagram $\Omega_{\Ho(\A_f)}(A)\stackrel{-\gamma_f(\xi_f)}\to \ker f\stackrel{\gamma_f(\iota_f)}\to B\stackrel{\gamma_f(f)}\to A$ is a distinguished left triangle in $\Ho(\A_f)$.
\end{proof}

Recall that the homotopy category $\Ho(\A)$ is a pretriangulated category in the sense of Definition 2.5 as proved by Quillen; see \cite[Chapter I]{Quillen67}.

\begin{corollary} \ The pretriangulated structure on $\A_{cf}/\omega$ given in Theorem 2.8 is equivalent to the one as constructed on $\Ho(\A)$ by Quillen.
\end{corollary}

\begin{proof} \  The functor $\A_{cf}/\omega\to \Ho(\A)$ by sending $A$ to $A$ and $\ul{f}:A\to B$ to $\gamma(f)$ is an equivalence of categories. This has been shown in \cite[Theorem VII.4.2]{Beligiannis/Reiten07}, \cite[Propositions 4.3,4.7]{Gillespie11} and \cite[Proposition 1.1.14]{Becker12}. So we only need to check that it is exact. Note that in $\Ho(\A)$, we have natural isomorphisms of functors $Q\Omega^f\cong \Omega^f$ and $\Sigma^c\cong R\Sigma^c$, thus the assertion follows from Proposition 4.9 directly since the pretriangulated structure on $\Ho(\A)$ is constructed via the one-sided triangulated structures of $\Ho(\A_f)$ and $\Ho(\A_c)$; see \cite[Theorem I.2.2, Proposition I.3.3, Proposition I.3.5]{Quillen67}. \end{proof}

\vskip10pt

\noindent{\bf Acknowledgments} \  The author would like to thank Xiao-Wu Chen for his stimulating discussions and encouragements. The author would especially like to thank Henning Krause for his great help. He thanks Xiaoxiang Yu, Guodong Zhou, Xiaojin Zhang, Yu Zhou for their helpful comments.

\vskip 10pt

 {\footnotesize \noindent Zhi-Wei Li \\
 School of Mathematics and Statistics,\\
  Jiangsu Normal University,
Xuzhou 221116, Jiangsu, PR China.\\
E-mail address: zhiweili$\symbol{64}$jsnu.edu.cn. }

\end{document}